\documentclass[12pt,twoside]{amsart}
\usepackage{amssymb,amsmath,amscd,enumerate,verbatim,xcolor,mathtools,fullpage}
\usepackage{enumitem}

\usepackage[top=3.5cm, bottom=2.5cm, left=3.0cm, right=3cm]{geometry}
\usepackage{color,colortbl, fancyhdr, graphicx, wrapfig}
\usepackage{multirow}
\usepackage{diagbox}
\usepackage[unicode]{hyperref}
\usepackage{indentfirst}
\usepackage{tikz}
\usepackage{array}
\usetikzlibrary{calc}
\usepackage{mathrsfs}
\usepackage{graphicx}
\usepackage{longtable}
\usepackage{cleveref}

\newtheorem{thm}{Theorem}[section]
\newtheorem{theorem}[thm]{Theorem}
\newtheorem{lemma}[thm]{Lemma}
\newtheorem{proposition}[thm]{Proposition}
\newtheorem{corollary}[thm]{Corollary}
\theoremstyle{definition}
\newtheorem{definition}[thm]{Definition}
\newtheorem{remark}[thm]{Remark}
\newtheorem{example}[thm]{Example}

\newtheorem{question}[thm]{Question}

\usepackage{hyperref}
\hypersetup{
	colorlinks=true,       
	linkcolor=red,          
	citecolor=blue,        
	filecolor=magenta,      
	urlcolor=cyan           
}

\newcommand{\avc}{A_{\operatorname{c}}}

\newcommand{\Cat}{\operatorname{Cat}}

\DeclareMathOperator{\Span}{span}
\DeclareMathOperator{\im}{Im}
\DeclareMathOperator{\Ker}{Ker}

\begin{document}

\title{Weak and strong Lefschetz properties for vertex cover Artinian algebras associated to graphs}
\author[T. Chau]{Trung Chau}
\address{Department of Mathematics, University of Manitoba, 420 Machray Hall, 186 Dysart  Road, Winnipeg, MB R3T 2N2 Canada}
\email{chauchitrung1996@gmail.com}

\author[T.Q. Hoa]{Tran Quang Hoa}
\address{University of Education, Hue University,  34 Le Loi St., Hue City, Vietnam.}
\email{tranquanghoa@hueuni.edu.vn}

\author[D.T.K. Uyen]{Dang Thi Kieu Uyen}
\address{University of Education, Hue University, 34 Le Loi St., Hue City, Vietnam.}
\email{dtkuyen@dhsphue.edu.vn}
\date{}



\begin{abstract}
Let $G$ be a finite simple graph and let $\avc(G)$ be the Artinian algebra associated with its cover ideal. We prove that $\avc(G)$ has the WLP when $\tau(G)>|V(G)|/2$, where $\tau(G)$ denotes the size of a minimum vertex cover of $G$. As a consequence, $A_c(G)$ has the WLP with high probability when  the Erd\H{o}s-R\'enyi random graph model is considered. Moreover, we study the borderline case $\tau(G)=|V(G)|/2$ and as a result, classify the WLP for paths, cycles, Ferrers graphs, and well-covered trees.
\end{abstract}
\maketitle
\section{Introduction}
Let \(A=\bigoplus_{i=0}^{D}[A]_i\) be a standard graded Artinian algebra over a field \(\Bbbk\). We say that \(A\) has the \emph{weak Lefschetz property} (WLP) if there exists a linear form \(\ell\in[A]_1\) such that, for every \(i\), the multiplication map
\[
\cdot\ell:[A]_i\longrightarrow[A]_{i+1}
\]
is either injective or surjective. Furthermore, \(A\) is said to have the \emph{strong Lefschetz property} (SLP) if there exists a linear form \(\ell\) such that, for every pair of integers \(i\geq 0\) and \(j\geq 1\), the multiplication map
\[
\cdot\ell^j:[A]_i\longrightarrow[A]_{i+j}
\]
has maximal rank.

The study of the WLP and SLP of Artinian algebra has been a central topic in commutatitve algebra, especially in the last two decades, due to its interesting connections to algebraic geometry, combinatorics, and representation theory. We refer to \cite{MN13} for a comprehensive survey of the topic.

Monomial algebras, in particular, have been a subject of intensive study \cite{AB20,BK11,CN12,KV14,LZ10,LN19,MMN11,MNS20}, partly due to their direct combinatorial connections. Moreover, Conca \cite{Conca2003} showed that the WLP of the algebra defined by the initial ideal of a homogeneous ideal, under certain circumstances, could imply  the WLP of the algebra defined by the original. Thus the understanding of the Lefschetz properties of monomial algebras has direct impact on the general case.

A special case of monomial algebras that we want to highlight is the class of Artinian algebras defined by edge ideals. Throughout this paper, all graphs are finite and simple and have at least one edge. Let \(G=(V,E)\) be a simple graph with vertex set \(V=\{1,2,\ldots,n\}\) and let
$R=\Bbbk[x_1,\ldots,x_n]$ be the standard graded polynomial ring over a field \(\Bbbk\). The \emph{edge ideal} of $G$ is
\[
I(G)\coloneqq (x_ix_j\mid \{i,j\}\in E(G)) \subseteq R.
\]
The Artinian monomial algebra $R/(I(G)+(x_1^2,\dots, x_n^2))$ and its Lefschetz properties have been studied by many \cite{DN24,HPS26a,HPS26,HN26,Kling24,MNS20,NT24b}. It is noteworthy that a full characterization of either the WLP or the SLP for this class is not known yet, underscoring the difficulty of determining these properties. In this article, we initiate the study of the dual version. Recall that the \emph{cover ideal} of \(G\) is defined to be
\[
J(G)\coloneqq\bigcap_{\{i,j\}\in E(G)}(x_i,x_j).
\]
Equivalently, $J(G)=
\big(x_C \mid C \text{ is a minimal vertex cover of }G\big),$
where $x_C=\prod_{i\in C}x_i.$
We define the Artinian algebra
\[
\avc(G)=\frac{R}{(x_1^2,\ldots,x_n^2)+J(G)},
\]
which is called the \emph{vertex cover Artinian algebra associated to \(G\)}. This leads naturally to the following problem.
\begin{question}
\label{quest_WLPofAvcG}
For which graphs \(G\) does the algebra \(\avc(G)\) have the WLP or the SLP? When \(\avc(G)\) fails to have one of these Lefschetz properties, in which degrees do the corresponding multiplication maps fail to have maximal rank?
\end{question}

A theme in the study of Lefschetz properties is how often an Artinian algebra has the WLP or the SLP. It is a  conjecture of Fr\"oberg \cite{Froberg1985} that $\mathbb{C}[x_1,\dots, x_n]/(f_1,\dots, f_m)$ has the SLP for generic forms $f_1,\dots, f_m$, with partial answers by Stanley \cite{S80} and Watanabe \cite{W87}. Another way to look at the conjecture is that: the SLP is expected to occur with high probability. On the other hand, for certain random models of Artinian monomial algebras, the WLP is expected to \textbf{not} occur with high probability \cite{Hol26,NP25}, the polar opposite. Our first result is that, limiting ourselves to some random models concerning cover ideals, we expect the WLP to occur. Let $\tau(G)$ denote the size of a minimum vertex cover of a graph $G$.

\begin{theorem}\label{thm:main1}
    The following holds.
    \begin{enumerate}
        \item If $G$ is a graph with $n$ vertices with $\tau(G)>n/2$, then $A_c(G)$ has the WLP. (Theorem~\ref{WLPtaularge})
        \item We have
        \[
        \lim_{n\to \infty}\mathbb{P}(A_c(G_{n,p}) \text{ has the WLP})=\lim_{n\to \infty}\mathbb{P}(A_c(\mathcal{G}_{n}) \text{ has the WLP})  = 1,
        \]
        where $G_{n,p}$ denotes the Erd\H{o}s-R\'enyi random graph on $n$ vertices and constant probability $p$, and $\mathcal{G}_n$ is a random graph isomorphism class on $n$ vertices with equal probability to all graphs. (Remarks~\ref{rem:Erdos-Renyi} and \ref{rem2:Erdos-Renyi})
    \end{enumerate}
\end{theorem}

We also consider many classes of graphs, and our results are mixed. Overall, the WLP does not seem to be common.

\begin{theorem}
    \begin{enumerate}[label=\textup{(\arabic*)}]
	\item For every $n\geq 2$, the algebra $\avc(K_n)$ has the SLP, where $K_n$ denotes the complete graph on $n$ vertices. (Proposition~\ref{prop:K_n})
    \item For the path graph $P_n$, the algebra $\avc(P_n)$ has the WLP if and only if $n$ is odd or $n\leq 5$. (Theorem~\ref{thm:path_classifition})
	\item For the cycle graph $C_n$, the algebra $\avc(C_n)$ has the WLP if and only if $n$ is odd. (Theorem~\ref{thm:cycle_classification})
    \item For a Ferrers graph $G$, the algebra $A_c(G)$ has the WLP if and only if $G\in \{K_2,P_4,C_4\}$. (Theorem~\ref{thm:ferrers_classification})
	\item If $G=H\circ K_1$ with $H$ connected bipartite and $|V(H)|\geq 3$, then $\avc(G)$ fails the WLP. (Corollary~\ref{cor:connected_bipartite_corona_failure})
	\item If $T$ is a nontrivial well-covered tree, then $\avc(T)$ has the WLP if and only if $T\cong K_2$ or $T\cong P_4$. (Corollary~\ref{cor:well_covered_trees})
	\item Let $G$ be a graph with vertex set $[2m]$ such that $\tau(G)=m$.  Assume that $G$ is bipartite and has a perfect matching. Set $\Cat_m=\frac{1}{m+1}\binom{2m}{m},$ the Catalan number.
	If $c_m(G)\leq \Cat_m$, where $c_m(G)$ is the number of minimum vertex covers of $G$, then $\avc(G)$ does not have the WLP. (Propositions~\ref{prop:complementary_failure} and \ref{prop:complementary_characterization})
    \end{enumerate}
\end{theorem}



The structure of the article is as follows. In Section~\ref{sec:prem}, we provide the necessary background in both algebra and graph theory, setting up the notations for the sequel. Our first main result, Theorem~\ref{thm:main1}, together with many partial answers on the borderline case $\tau(G)=n/2$, are proved in Section~\ref{sec:largetau}. A full characterization of the WLP for paths, cycles, and Ferrers graphs is presented in Sections~\ref{sec:paths} and \ref{sec:cyclesFerrers}. Finally, we study the WLP for some corona graphs in Section~\ref{sec:corona}.

\noindent\textbf{Acknowledgment.} Chau appreciates the support by the Infosys Foundation during his postdoc at Chennai Mathematical Institute, and the Pacific Institute for the Mathematical Sciences. Chau thanks Dang Hop Nguyen and the International Center for Research and Postgraduate Training in Mathematics (ICRTM) for funding his visit to the Institute of Mathematics, Vietnam Academy of Science and Technology (VAST) in June-July 2026, where this work was initiated. The first author thanks Thiago Holleben for his helpful feedback and comments. Part of this work was carried out while the second and third authors
were visiting the Vietnam Institute for Advanced Study in Mathematics
(VIASM). They thank VIASM for its hospitality and financial support.

\section{Preliminaries}\label{sec:prem}
\subsection{The weak Lefschetz property}

Let $R=\Bbbk[x_1,\ldots,x_n]$ be the standard graded polynomial ring over a field $\Bbbk$, where each $x_i$ has degree $1$, and let $I \subset R$ be a homogeneous ideal such that $A=R/I$ is Artinian. Then $A$ is a graded Artinian algebra, which can be decomposed as
\[
A=\bigoplus_{i=0}^{D}[A]_i.
\]
\begin{definition}
We say that $A$ has the \emph{weak Lefschetz property}
(WLP) if there exists a linear form $\ell\in[A]_1$ such
that, for every integer $j$, the multiplication map
\[
\cdot\ell:[A]_j\longrightarrow[A]_{j+1}
\]
has maximal rank, that is, it is injective or surjective.
Such an $\ell$ is called a \emph{weak Lefschetz element}.

We say that $A$ has the \emph{strong Lefschetz property}
(SLP) if there exists a linear form $\ell\in[A]_1$ such
that, for every integer $j$ and every $q\geq1$, the map
\[
\cdot\ell^q:[A]_j\longrightarrow[A]_{j+q}
\]
has maximal rank. Such an $\ell$ is called a
\emph{strong Lefschetz element}.
\end{definition}
Throughout this paper, we consider Artinian algebras defined by
monomial ideals. For such algebras, the sum of the variables
suffices to test both Lefschetz properties.

\begin{theorem}
[\cite{LN19,MMN11, N18}]\label{Theorem2.5}
Let $I \subset R$ be an Artinian monomial ideal and $\ell$ be the sum of all variables. Then $A=R/I$ has the WLP if and only if
$\ell$ is a weak Lefschetz element. Likewise, $A$ has the SLP
if and only if $\ell$ is a strong Lefschetz element.
\end{theorem}

Theorem~\ref{Theorem2.5} was first proved in \cite[Proposition~2.2]{MMN11} over an infinite field, and later in \cite[Proposition~4.3]{LN19} over an arbitrary field.  The SLP assertion also holds
over an arbitrary field; see \cite[Theorem~2.2]{N18}.

In view of Theorem~\ref{Theorem2.5}, we shall frequently work with multiplication by the linear form $\ell=x_1+\cdots+x_n$. Let $B_n=R/(x_1^2,\ldots,x_n^2)$ be the Boolean algebra. Note that $B_n$ has a natural $\Bbbk$-basis consisting of the squarefree monomials $x_A=\prod_{x_i\in A}x_i$, where $A\subseteq \{x_1,\ldots,x_n\}$. For each $d\geq 0$, multiplication by $\ell$ induces a linear map $U_d:[B_n]_d\longrightarrow [B_n]_{d+1}$, called the \textit{up-map}. On the squarefree monomial basis, it is given by $U_d(x_A)=\sum_{x_i\notin A}x_{A\cup\{x_i\}}$.

We define the monomial inner product on $B_n$ by
\[
\langle x_A,x_B\rangle=
\begin{cases}
	1, & \text{if } A=B,\\
	0, & \text{otherwise}.
\end{cases}
\]
With respect to this inner product, the adjoint of $U_d$ is the \textit{down-map} 
\[
D_{d+1}:[B_n]_{d+1}\longrightarrow [B_n]_d,
\]
given by $D_{d+1}(x_A)=\sum_{x_i\in A}x_{A\setminus\{x_i\}}$. Equivalently, $D=\partial/\partial x_1+\cdots+\partial/\partial x_n$ on squarefree monomials. Thus, for $f\in [B_n]_d$ and $g\in [B_n]_{d+1}$, one has $\langle U_d(f),g\rangle=\langle f,D_{d+1}(g)\rangle$.

\begin{theorem}[\cite{PT23,S80, W87}]\label{SLP_Stanley}
Assume that $\operatorname{char}(\Bbbk)=0$. The algebra $B_n=R/(x_1^2,\ldots,x_n^2)$ has the SLP. 
In particular, set $m=\lfloor \frac n2\rfloor$, the up-map
$U_d:[B_n]_d\longrightarrow [B_n]_{d+1}$ is injective for $d<m$ and surjective for $d\geq m$. Equivalently, the down-map
$D_{d+1}:[B_n]_{d+1}\longrightarrow [B_n]_d$ is surjective for $d<m$ and injective for $d\geq m$.
\end{theorem}

We now return to general Artinian graded algebras and recall a basic necessary condition for the WLP. To state it, we first introduce the Hilbert series and the notion of unimodality.

\begin{definition}
Let $\Bbbk$ be a field and $A= \bigoplus_{j\geq 0} [A]_j$ be a standard graded $\Bbbk$-algebra. The \emph{Hilbert series} of $A$ is the power series 
\[	
H_A(t) = \sum_{j\geq 0} \dim_\Bbbk [A]_j \, t^j.
\]
If $A$ is Artinian, then $[A]_i=0$ for $i \gg 0$. We denote 
\[
e := \max \{i \mid [A]_i \neq 0\},
\]
and we call $e$ the \emph{socle degree} of $A$. 
In this case, the Hilbert series of $A$ is a polynomial
\[
H_A(t) = 1 + h_1 t + \cdots + h_e t^e,
\]	
where $h_i = \dim_\Bbbk [A]_i > 0$. 
\end{definition}
\begin{definition}
\label{mode}
A polynomial $\sum_{k=0}^n a_k t^k$ with non-negative coefficients is called \emph{unimodal} if there exists an integer $m$ such that 
\[
a_0 \leq a_1 \leq \cdots \leq a_{m-1} \leq a_m \geq a_{m+1} \geq \cdots \geq a_n.
\]
Set $a_{-1}=0$. The \emph{mode} of a unimodal polynomial $\sum_{k=0}^n a_k t^k$ is defined to be the unique integer $i \in \{0,\dots,n\}$ such that 
\[
a_{i-1} < a_i \geq a_{i+1} \geq \cdots \geq a_n.
\]
\end{definition}

\begin{example}
The polynomial $5 + 5t + 5t^2 + 5t^3 + 5t^4$ is unimodal of mode $0$, while the polynomial $1 + 5t + 5t^2 + 5t^3 + 5t^4$ is unimodal of mode $1$.
\end{example}

\begin{proposition}
[{\cite[Proposition~3.2]{HMMNWW13}}]\label{Proposition2.8}
If $A$ has the WLP, then the Hilbert series of $A$ is unimodal.
\end{proposition}

\subsection{Graph theory}

Throughout this subsection, by a \emph{graph} we mean a simple, undirected graph $G=(V,E)$ with vertex set $V=V(G)$ and edge set $E=E(G)$. 
We begin by recalling some basic definitions.
\begin{definition}
The \emph{disjoint union} of two graphs $G_1$ and $G_2$ is the graph $G = G_1 \cup G_2$ whose vertex set is the disjoint union of $V(G_1)$ and $V(G_2)$, and whose edge set is the disjoint union of $E(G_1)$ and $E(G_2)$. 
\end{definition}
\begin{definition}
Let $G=(V,E)$ be a graph with vertex set $V=\{1,2,\dots,n\}$.
\begin{itemize}
\item[(a)] A subset $X \subset V$ is called an \emph{independent set} if, for all $i,j \in X$, we have $\{i,j\} \notin E$. An independent set $X$ of cardinality $k$ is called a \emph{$k$-independent set} of $G$.	 
\item[(b)] An independent set $X$ is called a maximal independent set if, for every $v\in V(G)\setminus X$, the set $X\cup\{v\}$ is not independent.
\item[(c)] The \emph{independence number} of $G$ is the maximum cardinality of an independent set in $G$, denoted by $\alpha(G)$. 
\item[(d)] The \emph{independence polynomial} of $G$ is the polynomial in the variable $t$ whose coefficient of $t^k$ is the number of $k$-independent sets. We denote it by $I(G;t)$, that is,
\[
I(G;t)= \sum_{k=0}^{\alpha(G)} s_k(G)t^k,
\]
where $s_k(G)$ denotes the number of $k$-independent sets of $G$.
\end{itemize}
\end{definition}

The independence polynomial of a graph was defined by Gutman and Harary in \cite{GH83} as a generalization of the matching polynomial of a graph.

\begin{definition} 
Let $G=(V,E)$ be a graph with vertex set $V=\{1,2,\dots,n\}$.
\begin{enumerate}
\item [(a)] A subset $C\subseteq V$ is a \emph{vertex cover} of $G$ if every edge of $G$ has at least one endpoint in $C$. Equivalently,
\[
\forall \{u,v\}\in E\quad\Longrightarrow\quad u\in C\text{ or }v\in C.
\]
\item [(b)]  A vertex cover $C\subseteq V(G)$ is called a \emph{minimal vertex cover} if no proper subset of $C$ is a vertex cover of $G$. Equivalently, for every $v\in C$, the set $C\setminus\{v\}$ is not a vertex cover of $G$.
\item [(c)] The \emph{vertex-cover number} of $G$, denoted by $\tau(G)$, is the minimum cardinality of a vertex cover of $G$.
\end{enumerate}
\end{definition}
\begin{definition} 
Let $G=(V,E)$ be a graph with vertex set $V=\{1,2,\dots,n\}$. Let $R=\Bbbk[x_1,\ldots,x_n]$.
\begin{enumerate}
\item [(a)] The \emph{cover ideal} of $G$ is
\[
J(G)=(x_C:C\text{ is a vertex cover of }G)\subset R,
\]
where $x_C=\prod_{i\in C}x_i$. Equivalently, $J(G)=\bigcap_{\{u,v\}\in E}(x_u,x_v).$
\item [(b)] The \emph{vertex cover Artinian algebra} associated to $G$ is
\[
\avc(G)=\frac{R}{\big((x_1^2,\ldots,x_n^2)+J(G)\big)}.
\]
\end{enumerate}
\end{definition}

\begin{proposition}\label{prop:HilbertseriesofG}
For every $d\ge 0$,
\[
[\avc(G)]_d=\operatorname{span}_{\Bbbk}\{x_A: |A|=d,\ A\text{ is not a vertex cover of }G\}.
\]
In particular, if $c_d(G)$ denotes the number of vertex covers of $G$ of cardinality $d$, then
\[
\dim_{\Bbbk}[\avc(G)]_d=\binom{n}{d}-c_d(G).
\]
\end{proposition}

\begin{proof}
In the quotient by the squares of the variables, the remaining monomials are exactly the squarefree monomials $x_A$. Such a monomial is zero in $\avc(G)$ if and only if it lies in $J(G)$, which is equivalent to $A$ being a vertex cover.
\end{proof}

A subset $C\subseteq V$ is a vertex cover if and only if its complement $V\setminus C$ is an independent set. Therefore, one has $c_d(G)=s_{n-d}(G).$
Consequently,
\begin{align*}
H_{\avc(G)}(t)&=(1+t)^{n}-\sum_{d\ge 0}c_d(G)t^d\\
&=(1+t)^{n}-\sum_{d\ge 0}s_{n-d}(G)t^{d}.
\end{align*}
Throughout this article let $[n]$ denote the set $\{1,2,\dots, n\}$ for each integer $n$.

\begin{lemma}\label{lem:upset_shadow_growth}
Let $\mathcal F\subseteq 2^{[n]}$ be an up-set, that is, if $A\in \mathcal F$ and
$A\subseteq B\subseteq [n]$, then $B\in \mathcal F$. For each $r$, set
\[
\mathcal F^{(r)}=\{A\in \mathcal F: |A|=r\}.
	\]
	Then
$|\mathcal F^{(r-1)}| \le |\mathcal F^{(r)}|$, for all $1\leq r\leq \left\lceil\frac n2\right\rceil.$
\end{lemma}

\begin{proof}
	Fix an integer $r$ with $1\leq r\leq \left\lceil n/2\right\rceil$.
We define the upper shadow of $\mathcal F^{(r-1)}$ by
\[
\nabla \mathcal F^{(r-1)}
=
\left\{
A\cup\{x\}: A\in \mathcal F^{(r-1)},\ x\in [n]\setminus A
\right\}.
\]
Since $\mathcal F$ is an up-set, every set in $\nabla \mathcal F^{(r-1)}$ belongs to
	$\mathcal F$. Hence $	\nabla \mathcal F^{(r-1)}\subseteq \mathcal F^{(r)}.$
	Therefore it is enough to show that
	\[
	|\nabla \mathcal F^{(r-1)}|\geq |\mathcal F^{(r-1)}|.
	\]
We construct a bipartite graph $\Gamma$ as follows. Its left vertex set is $\mathcal F^{(r-1)}$, and its right vertex set is $\nabla \mathcal F^{(r-1)}$. We join $A\in \mathcal F^{(r-1)}$ to $B\in \nabla \mathcal F^{(r-1)}$ if and only if $A\subseteq B.$
	
	We count the edges of $\Gamma$ in two ways. 
	First, fix $A\in\mathcal F^{(r-1)}$. Since $|A|=r-1$, there are exactly
	$n-r+1$ elements in $[n]\setminus A$. For each $x\in [n]\setminus A$, the set
	$A\cup\{x\}$ has cardinality $r$ and contains the set $A\in\mathcal F^{(r-1)}$.
	Therefore, by the definition of the upper shadow,
	\[
	A\cup\{x\}\in \nabla\mathcal F^{(r-1)}.
	\]
	Thus $A$ is adjacent in $\Gamma$ to exactly the sets $A\cup\{x\}$, where
	$x\in [n]\setminus A$. Hence the degree of $A$ in $\Gamma$ is exactly $n-r+1$.
 Hence the number of edges of $\Gamma$ is
	\[
	e(\Gamma)=(n-r+1)|\mathcal F^{(r-1)}|.
	\]
	
	On the other hand, fix $B\in \nabla \mathcal F^{(r-1)}$. Since $|B|=r$, the set $B$
	has exactly $r$ subsets of cardinality $r-1$. Hence $B$ can be adjacent to at most
	$r$ vertices of $\mathcal F^{(r-1)}$. Therefore
	\[
	e(\Gamma)\leq r|\nabla \mathcal F^{(r-1)}|.
	\]
	Combining the two estimates gives
	\[
	(n-r+1)|\mathcal F^{(r-1)}|
	\leq
	r|\nabla \mathcal F^{(r-1)}|.
	\]
	Since $r\leq \left\lceil n/2\right\rceil$, we have $n-r+1\geq r$. Hence $|\mathcal F^{(r-1)}| \leq|\nabla \mathcal F^{(r-1)}|.$
Together with $\nabla \mathcal F^{(r-1)}\subseteq \mathcal F^{(r)}$, this gives
	\[
	|\mathcal F^{(r-1)}|
	\leq
	|\nabla \mathcal F^{(r-1)}|
	\leq
	|\mathcal F^{(r)}|
	\]
	This proves the assertion.
\end{proof}

\begin{proposition}\label{pro:unimodal_last}
Let $G$ be a graph on $n$ vertices, and set $m=\lfloor n/2\rfloor$. If 	we write $H_{\avc(G)}(t)=\sum_{d= 0}^nh_dt^d$, then
\[
h_m\ge h_{m+1}\ge h_{m+2}\ge\cdots \ge h_n.
\]
In other words, the Hilbert function of $\avc(G)$ is non-increasing from degree $m$.
\end{proposition}

\begin{proof}
Let $\mathcal F$ be the family of all subsets of $V(G)$ containing at least one edge of $G$. 
For each $r$, set \[ \mathcal F^{(r)}=\{A\in\mathcal F: |A|=r\}. \] 
A subset $C\subseteq V(G)$ is not a vertex cover of $G$ if and only if its complement $V(G)\setminus C$ contains an edge of $G$. Therefore, by Proposition~\ref{prop:HilbertseriesofG}, $h_d=\mathcal F^{(n-d)},$ for every $0\leq d\leq n$.
	
We claim that \[ |\mathcal F^{(r-1)}|\leq |\mathcal F^{(r)}| \qquad\text{for all }1\leq r\leq \left\lceil\frac n2\right\rceil. \] Indeed, $\mathcal F$ is an up-set: if a set contains an edge of $G$, then every larger set also contains an edge of $G$. Hence the claim follows from Lemma~\ref{lem:upset_shadow_growth}. Now let $m\leq d\leq n-1$. Then \[ 1\leq n-d\leq \left\lceil\frac n2\right\rceil. \] Applying the claim with $r=n-d$, we get \[ |\mathcal F^{(n-d-1)}|\leq |\mathcal F^{(n-d)}|. \] Hence \[ h_d=|\mathcal F^{(n-d)}| \geq |\mathcal F^{(n-d-1)}| = h_{d+1}. \] Therefore \[ h_m\geq h_{m+1}\geq h_{m+2}\geq\cdots\geq h_n. \] 
This proves the assertion.
\end{proof}

We end this section with a consequence on the unimodality of the Hilbert series of $\avc(G)$ when the vertex cover number of $G$ is large.
\begin{corollary}\label{cor:unimodal_tau=n/2}
Let $G$ be a graph on $n$ vertices, and set $m=\lfloor n/2\rfloor$. Assume that $\tau(G)\geq m$. Then the Hilbert series of $\avc(G)$ is unimodal. Moreover, its mode is either $m$ or $m-1$.
\end{corollary}

\begin{proof}
	Write
	\[
	H_{\avc(G)}(t)=\sum_{d= 0}^nh_dt^d.
	\]
	Since $\tau(G)\geq m$, the cover ideal $J(G)$ has no nonzero component in degrees less than $m$. Hence $[\avc(G)]_d=[B_n]_d$ for all $d<m$.
	Therefore $h_d=\binom{n}{d}$ for all $d<m$. Since $m=\lfloor n/2\rfloor$, we get
	\[
	h_0<h_1<\cdots<h_{m-1}.
	\]
	
By Proposition~\ref{pro:unimodal_last},
	\[
	h_m\geq h_{m+1}\geq h_{m+2}\geq \cdots \ge h_n.
	\]
	
Thus the only possible change in monotonicity occurs between $h_{m-1}$ and $h_m$. Consequently, the Hilbert series is unimodal, and its mode is either $m-1$ or $m$.
\end{proof}

\section{Graphs with large vertex cover number}\label{sec:largetau}
From now on, we assume that $\operatorname{char}(\Bbbk)=0$. Moreover, whenever $G$ has vertex set $[n]$, we write $\ell=x_1+\cdots+x_n$.
\begin{theorem}\label{WLPtaularge}
Let \(G\) be a graph on \(n\) vertices. Assume that $\tau(G)>\frac n2,$ where \(\tau(G)\) denotes the vertex cover number of \(G\). Then $\avc(G)$ has the weak Lefschetz property.
\end{theorem}

\begin{proof}
Let $B_n=R/\left(x_1^2,\ldots,x_n^2\right)$ be  the Boolean algebra. 
Then, $\avc(G)=\frac{B_n}{\overline{J(G)}},$ where \(\overline{J(G)}\) is the image of \(J(G)\) in \(B_n\).
Since every minimal generator of \(J(G)\) has degree at least \(\tau(G)\), and since $\tau(G)>\frac n2,$  we get
$[\overline{J(G)}]_i=0$ for every $i\leq \left\lfloor \frac n2\right\rfloor.$
Hence $[\avc(G)]_i=[B_n]_i$ for all $i\leq \left\lfloor \frac n2\right\rfloor.$

By Theorem~\ref{SLP_Stanley}, $B_n$ has the strong Lefschetz property. In particular, the map
\[
U_i:[B_n]_i\longrightarrow [B_n]_{i+1}
\]
is injective for $i<\left\lfloor \frac n2\right\rfloor$
and is surjective for $i\geq \left\lfloor \frac n2\right\rfloor.$

Set $m=\left\lfloor \frac n2\right\rfloor.$
If \(i<m\), then \(i+1\leq m\). It follows that $[\avc(G)]_i=[B_n]_i$ and $[\avc(G)]_{i+1}=[B_n]_{i+1}.$
Therefore the multiplication map
\[
\cdot\ell:[\avc(G)]_i\longrightarrow [\avc(G)]_{i+1}
\]
is identified with the up-map
\[
U_i:[B_n]_i\longrightarrow [B_n]_{i+1},
\]
which is injective.

If \(i\geq m\), then the map $U_i:[B_n]_i\longrightarrow [B_n]_{i+1}$
is surjective. Passing to the quotient by \(\overline{J(G)}\), it follows that
\[
\cdot \ell:[\avc(G)]_i\longrightarrow [\avc(G)]_{i+1}
\]
is also surjective.

Thus multiplication by \(\ell\) has maximal rank in every degree. Hence $\avc(G)$
has the weak Lefschetz property.
\end{proof}

\begin{remark}
The same argument does not prove the strong Lefschetz property.

Indeed, it proves only the following partial facts. We always have that
the maps
\[
\cdot \ell^s:[\avc(G)]_i\longrightarrow [\avc(G)]_{i+s}
\]
are injective if $i+s\leq \left\lfloor \frac n2\right\rfloor,$
and are surjective if $i\geq \left\lfloor \frac n2\right\rfloor.$
However, the maps crossing the middle degree are not controlled by this argument.
In fact, the strong Lefschetz property may fail, in general.
\end{remark}

\begin{example}
Let $G=K_4\sqcup K_2$ be the graph on the vertex set $V(G)=\{1,2,3,4,5,6\}$, where $\{1,2,3,4\}$ spans a complete graph $K_4$, and $\{5,6\}$ is an additional edge.
Then $\tau(G)=4 >3=\frac n2.$
So, by Theorem~\ref{WLPtaularge}, \(\avc(G)\) has the weak Lefschetz property.

We show that \(\avc(G)\) does not have the strong Lefschetz property. The minimal vertex covers of \(G\) are obtained by taking three vertices from $\{1,2,3,4\}$ and one vertex from $\{5,6\}.$
Hence \(J(G)\) is generated by the eight monomials $x_ix_jx_kx_s,$
where $\{i,j,k\}\subseteq \{1,2,3,4\}$ and $s\in\{5,6\}.$
Therefore, a direct computation with Macaulay2~\cite{GS} gives
\[
H_{\avc(G)}(t)=1+6t+15t^2+20t^3+7t^4.
\]

It follows that if $\avc(G)$ had the SLP, then, for a general linear form $L\in [\avc(G)]_1$, the map
\[
\cdot L^3:[\avc(G)]_1\longrightarrow [\avc(G)]_4
\]
would have maximal rank. Since $\dim_{\Bbbk}[\avc(G)]_1=6$ and $\dim_{\Bbbk}[\avc(G)]_4=7,$
this map would have to be injective.

Let $L=a_1x_1+\cdots+a_6x_6\in [\avc(G)]_1$ be a general linear form. Since $L$ is general, we may assume that $a_1a_2\cdots a_6\neq 0$. Then $u=a_6x_6-a_5x_5$
is a nonzero element of $[\avc(G)]_1$. We claim that $L^3u=0$ in $\avc(G)$.

Indeed, all terms of $a_6L^3x_6$ whose support contains three variables among $x_1,x_2,x_3,x_4$ are multiples of a generator of $J(G)$, and hence
vanish in $\avc(G)$. Similarly, all terms of $a_5L^3x_5$ whose support contains three variables among $x_1,x_2,x_3,x_4$ vanish in $\avc(G)$.
The remaining possible terms have support of the form $\{x_i,x_j,x_5,x_6\},$ with  $1\leq i<j\leq 4.$
For such a monomial, the coefficient in $a_6L^3x_6$ is $6a_ia_ja_5a_6,$
while the coefficient in $a_5L^3x_5$ is also $6a_ia_ja_5a_6.$
Since
\[
L^3u=L^3(a_6x_6-a_5x_5)=a_6L^3x_6-a_5L^3x_5,
\]
these two contributions cancel. Therefore $L^3u=0$ in $\avc(G)$. It follows that the map
\[
\cdot L^3:[\avc(G)]_1\longrightarrow [\avc(G)]_4
\]
has a nonzero kernel. Hence it is not injective. Consequently, $\avc(G)$ does not
have the SLP.
\end{example}

The condition $\tau(G)>\frac{n}{2}$ seems restrictive, but in fact is often satisfied in many random graph models. We shall take a few examples.

\begin{remark}\label{rem:Erdos-Renyi}
    Let $G_{n,p}$ denote a graph with vertex set $[n]$ where any two vertices have probability $p$ to be adjacent, where $p\in [0,1]$ is a constant. This is the classic Erd\H{o}s-R\'enyi random graph model \cite{ER59}. Matula \cite{Mat70,Mat72} and independently Bollob\'as and P. Erd\H{o}s \cite{BE76} showed that for most $n$, we have
    \[
    \alpha(G_{n,p})\approx 2\log_{1/(1-p)} n<\frac{n}{2} \text{ with high probability (whp).}
    \]
    Thus 
    \[
    \tau(G_{n,p})=n-\alpha(G_{n,p}) >\frac{n}{2} \text{ whp}.
    \]
    Consequently, $G_{n,p}$ has the WLP whp by Theorem~\ref{WLPtaularge}.
\end{remark}

\begin{remark}\label{rem2:Erdos-Renyi}
    In Remark~\ref{rem:Erdos-Renyi}, $G_{n,p}$ is a \emph{labeled} graph, in the sense that up to isomorphisms, the Erd\H{o}s-R\'enyi random graph model attributes different probability to different graphs. A simpler model is $\mathcal{G}_{n}$, the random graph model that assigns the same probability to the class of isomorphisms of graphs on vertex set $[n]$. In other words, $\mathcal{G}_n$ is the \emph{unlabeled} version of $G_{n,1/2}$. Then since for most $n$, the graph $G_{n,1/2}$ has the WLP whp, the same holds for $\mathcal{G}_n$ (\cite[p. 1462]{Bab95}).
\end{remark}

Let $K_n$ be the complete graph on $[n]$, where $n\geq 2$. A subset $C\subseteq [n]$ is a vertex cover of $K_n$ if and only if its complement has cardinality at most $1$. Therefore $\tau(K_n)=n-1$. Since $n-1>n/2$ for every $n\geq 3$, Theorem~\ref{WLPtaularge} implies that $\avc(K_n)$ has the WLP. In fact, we have the following stronger result.

\begin{proposition}\label{prop:K_n}
The algebra $\avc(K_n)$ has the SLP.
\end{proposition}
\begin{proof}
Let $B_n=R/(x_1^2,\ldots,x_n^2)$. Since the minimal vertex covers of $K_n$ are precisely the subsets of $[n]$ of cardinality $n-1$, we have
$J(K_n)=(x_1\cdots \widehat{x_i}\cdots x_n \mid 1\leq i\leq n)$.
Hence $\avc(K_n)$ agrees with $B_n$ in all degrees at most $n-2$, and $[\avc(K_n)]_d=0$ for all $d\geq n-1$.
	
Let $k\geq 0$ and $i\geq 1$. We consider the multiplication map
\[\cdot \ell^i:[\avc(K_n)]_k\longrightarrow [\avc(K_n)]_{k+i}.\]
	
If $k+i\leq n-2$, then both graded components of $\avc(K_n)$ coincide with the corresponding graded components of $B_n$. Therefore the above map is identified with
$\cdot \ell^i: [B_n]_k\longrightarrow [B_n]_{k+i}$.
By Theorem~\ref{SLP_Stanley}, $B_n$ has the SLP in characteristic zero. Hence this map has maximal rank.
	
If $k+i\geq n-1$, then $[\avc(K_n)]_{k+i}=0$. Thus the map
\[
\cdot \ell^i:[\avc(K_n)]_k\longrightarrow [\avc(K_n)]_{k+i}
\]
is automatically surjective, and hence has maximal rank.
Therefore multiplication by every power of $\ell$ has maximal rank in every degree. Thus $\avc(K_n)$ has the SLP, as desired.
\end{proof}

We now consider the borderline case where the vertex cover number is equal to $\lfloor n/2\rfloor$.

\begin{proposition}\label{pro:critical_map}
Assume that  $\tau(G)=m=\lfloor n/2\rfloor$. Then $\avc(G)$ has the WLP if and only if the multiplication map
$\cdot \ell : [\avc(G)]_{m-1}\longrightarrow [\avc(G)]_m$
has maximal rank.
\end{proposition}

\begin{proof}
As in the proof of Theorem 3.1, set $B_n=R/(x_1^2,\ldots,x_n^2)$. Since $\tau(G)=m$, the ideal $J(G)$ has no nonzero component in degrees less than $m$. Hence $[\avc(G)]_i=[B_n]_i$ for every $i<m$.
	
By Theorem~\ref{SLP_Stanley}, $B_n$ has the SLP in characteristic zero. In particular, the up-maps $U_i:[B_n]_i\longrightarrow [B_n]_{i+1}$ are injective for $i<m$ and surjective for $i\geq m$.
	
Let $i\leq m-2$. Then $[\avc(G)]_i=[B_n]_i$ and $[\avc(G)]_{i+1}=[B_n]_{i+1}$. Thus the map 
\[
\cdot \ell:[\avc(G)]_i\longrightarrow [\avc(G)]_{i+1}
\]
is identified with the up-map $U_i:[B_n]_i\longrightarrow [B_n]_{i+1}$, and hence is injective. Therefore it has maximal rank for all $i\leq m-2$.
	
Now let $i\geq m$. Since $U_i:[B_n]_i\longrightarrow [B_n]_{i+1}$ is surjective, the induced map on the quotient
\[
\cdot \ell:[\avc(G)]_i\longrightarrow [\avc(G)]_{i+1}
\]
is also surjective. Hence it has maximal rank for all $i\geq m$.
	
Thus all Lefschetz maps of $\avc(G)$ have maximal rank except possibly the single map
\[
\cdot \ell:[\avc(G)]_{m-1}\longrightarrow [\avc(G)]_m.
\]
Consequently, $\ell$ is a weak Lefschetz element for $\avc(G)$ if and only if this middle map has maximal rank, as desired.
\end{proof}

\begin{remark}
In the previous proposition, the phrase ``maximal rank'' cannot be replaced by ``injective''. Indeed, for the critical map $\cdot \ell:[\avc(G)]_{m-1}\longrightarrow [\avc(G)]_m$, the expected maximal rank behavior depends on the two dimensions $\dim_\Bbbk[\avc(G)]_{m-1}$ and $\dim_\Bbbk[\avc(G)]_m$. Since $\tau(G)=m$, $\dim_\Bbbk[\avc(G)]_{m-1}=\binom{n}{m-1}$. On the other hand, one has 
\[
\dim_\Bbbk[\avc(G)]_m=\binom{n}{m}-\#\{C\subseteq V\mid C \text{ is a minimal vertex cover of }G\text{ and } |C|=m\}.
\]
Therefore it may happen that $\dim_\Bbbk[\avc(G)]_{m-1}>\dim_\Bbbk[\avc(G)]_m$. In this case the critical map cannot be injective, and maximal rank means surjectivity. 
\end{remark}
\begin{example}\label{exam:pathP4} 
Let $P_4$ be the path on $[4].$ Then the Hilbert series of $\avc(P_4)$ is
	\[
	H_{\avc(P_4)}(t)=1+4t+3t^2.
	\]
Set $\ell=x_1+x_2+x_3+x_4$. The multiplication map
\[
\cdot \ell:[\avc(P_4)]_1\longrightarrow [\avc(P_4)]_2
\]
is surjective. Indeed, with respect to the bases of the source and the target
\[
[\avc(P_4)]_1=
\langle x_1,x_2,x_3,x_4\rangle\quad\text{and}\quad
[\avc(P_4)]_2=
\langle x_1x_2,x_1x_4,x_3x_4\rangle,
\]
the multiplication map by $\ell$ is
\[
\begin{pmatrix}
1&1&0&0\\
1&0&0&1\\
0&0&1&1
\end{pmatrix}.
\]
This matrix has rank $3$, so $\cdot\ell$ is surjective. Hence $\avc(P_4)$ has WLP in any characteristic of $\Bbbk$.
\end{example}

\begin{corollary}
Let $G$ be a graph with vertex set $[n]$ and set $m=\lfloor n/2\rfloor$.  Assume that $\tau(G)=m$.  Let
\[
M_G=\Span_{\Bbbk}\{x_C:C\subseteq [n],\ |C|=m,\ C\text{ is a vertex cover of }G\}\subseteq [B_n]_m.
\]
Let $\pi_G:[B_n]_m\to [B_n]_m/M_G$ be the quotient map.  Then $\avc(G)$ has the WLP if and only if
$\ell=\pi_G\circ U_{m-1}:[B_n]_{m-1}\longrightarrow [B_n]_m/M_G$
has maximal rank. More precisely, one has
\begin{enumerate}[label=\textup{(\arabic*)}]
	\item If
	$\binom n{m-1}\leq \binom nm-c_m(G),$
	then $\avc(G)$ has the WLP if and only if
	$\im U_{m-1}\cap M_G=0.$
	\item If
	$\binom n{m-1}\geq \binom nm-c_m(G),$
	then $\avc(G)$ has the WLP if and only if
	$\im U_{m-1}+M_G=[B_n]_m.$
\end{enumerate}
\end{corollary}
\begin{proof}
	The map $q_G\circ U_{m-1}$ has kernel
	$\Ker(q_G\circ U_{m-1})=U_{m-1}^{-1}(\im U_{m-1}\cap M_G).$
	In the first case, maximal rank means injectivity.  Since $U_{m-1}$ is injective in the Boolean algebra, injectivity of $q_G\circ U_{m-1}$ is equivalent to $\im U_{m-1}\cap M_G=0$.
	
	In the second case, maximal rank means surjectivity.  The image of $q_G\circ U_{m-1}$ is
	$q_G(\im U_{m-1})=(\im U_{m-1}+M_G)/M_G.$
	This equals $[B_n]_m/M_G$ if and only if $\im U_{m-1}+M_G=[B_n]_m$.
\end{proof}

If $n=2m$ is even, then
$\binom{2m}{m}-\binom{2m}{m-1}=\frac{1}{m+1}\binom{2m}{m}.$
We denote this Catalan number by
$\Cat_m=\frac{1}{m+1}\binom{2m}{m}.$

\begin{corollary}\label{cor:even_borderline_cat}
Let $G$ be a graph with vertex set $[2m]$ with $\tau(G)=m$.  Then
\[
\dim[\avc(G)]_m-\dim[\avc(G)]_{m-1}=\Cat_m-c_m(G).
\]
	In particular, if $c_m(G)\leq \Cat_m$, then the WLP is equivalent to
	$\im U_{m-1}\cap M_G=0.$
	If $c_m(G)\geq \Cat_m$, then the WLP is equivalent to
	$\im U_{m-1}+M_G=[B_{2m}]_m.$
\end{corollary}

\begin{lemma}\label{lem:two_block_kernel}
Let $V=O\sqcup E$ be a disjoint union with $|O|=|E|=m$.  For $r=0,1,\ldots,m-1$, set
\[
\Sigma_r=\sum_{\substack{|A|=m-1\\ |A\cap O|=r}}x_A.
\]
Thus $\Sigma_r$ is the sum of all squarefree monomials of degree $m-1$ involving exactly $r$ variables from $O=\{1,3,\ldots,2m-1\}$. 
Define
\[
f=\sum_{r=0}^{m-1}\frac{(-1)^{m-1-r}}{m\binom{m-1}{r}}\Sigma_r.
\]
Then
$U_{m-1}(f)=x_O+(-1)^{m-1}x_E$ in $[B_{2m}]_m$.
\end{lemma}
\begin{proof}
Write 
\[
a_r=\frac{(-1)^{m-1-r}}{m\binom{m-1}{r}}\quad \text{and hence}\quad
f=\sum_{r=0}^{m-1}a_r\Sigma_r \in [B_{2m}]_{m-1}.\] 

Let $T\subset [2m]$ with $|T|=m$, and put $r=|T\cap O|$. We compute the coefficient of $x_T$ in $U_{m-1}(f)$. The monomial $x_T$ can arise from $x_{T\setminus\{i\}}$ for some $i\in T$. If $i\in O$, then $T\setminus\{i\}$ contains $r-1$ elements from $O$, and there are $r$ such choices of $i$. If $i\notin O$, then $T\setminus\{i\}$ contains $r$ elements from $O$, and there are $m-r$ such choices of $i$. Hence the coefficient of $x_T$ in $U_{m-1}(f)$ is
\[
r a_{r-1}+(m-r)a_r.
\]We claim that this number is zero for all $1\le r\le m-1$. Indeed, for all $1\le r\le m-1$, one has
\[
a_{r-1}
=
\frac{(-1)^{m-r}}{m\binom{m-1}{r-1}}
\quad
\text{and} \quad
a_r
=
\frac{(-1)^{m-1-r}}{m\binom{m-1}{r}}.
\]
Hence
\begin{align*}
r a_{r-1}+(m-r)a_r&=
\frac{r(-1)^{m-r}}{m\binom{m-1}{r-1}} + \frac{(m-r)(-1)^{m-1-r}}{m\binom{m-1}{r}}\\
&= 	\frac{r}{m\binom{m-1}{r-1}} \big((-1)^{m-r}+(-1)^{m-1-r}\big)\\
&=0,
\end{align*}
since
\[
\binom{m-1}{r}
=
\frac{m-r}{r}\binom{m-1}{r-1}, \quad\text{and hence}\quad \frac{r}{\binom{m-1}{r-1}}
=
\frac{m-r}{\binom{m-1}{r}}.
\]

It remains to consider the two boundary cases. If $r=m$, then $T=O$, and the coefficient of $x_O$ in $U_{m-1}(f)$ is
\[
m a_{m-1}
=
m\cdot \frac{1}{m\binom{m-1}{m-1}}
=
1.
\]
If $r=0$, then $T=E=[2m]\setminus O$, and the coefficient of $x_E$ in $U_{m-1}(f)$ is
\[
m a_0
=
m\cdot \frac{(-1)^{m-1}}{m\binom{m-1}{0}}
=
(-1)^{m-1}.
\]
Consequently,
\[
U_{m-1}(f)=x_O+(-1)^{m-1}x_E,
\]
as claimed.
\end{proof}

\begin{proposition}\label{prop:complementary_failure}
Let $G$ be a graph with vertex set $[2m]$ such that $\tau(G)=m$.  Assume that $G$ has two complementary minimum vertex covers $O$ and $E$, that is,
\[
O\cap E=\emptyset,
\qquad
O\cup E=[2m],
\qquad
|O|=|E|=m.
\]
If
$c_m(G)\leq \Cat_m,$
then $\avc(G)$ does not have the WLP.
\end{proposition}

\begin{proof}
	By Corollary~\ref{cor:even_borderline_cat}, the assumption $c_m(G)\leq \Cat_m$ means that maximal rank for the critical map is equivalent to injectivity.  Hence it is enough to find a nonzero element in the kernel of
	$\cdot\ell:[\avc(G)]_{m-1}\longrightarrow [\avc(G)]_m.$
	Since $m-1<\tau(G)$, no squarefree monomial of degree $m-1$ is killed by $J(G)$.  Thus $[\avc(G)]_{m-1}=[B_{2m}]_{m-1}$.
	
	Apply Lemma~\ref{lem:two_block_kernel} to the decomposition $[2m]=O\sqcup E$.  We get a nonzero element $f\in [B_{2m}]_{m-1}$ such that
	$U_{m-1}(f)=x_O+(-1)^{m-1}x_E.$
	Since $O$ and $E$ are vertex covers of cardinality $m$, both monomials $x_O$ and $x_E$ vanish in $\avc(G)$.  Therefore
	$\ell f=0$ in $[\avc(G)]_m.$
	On the other hand, $f\ne 0$ in $[\avc(G)]_{m-1}$, because no monomial of degree $m-1$ is killed.  Hence the critical map is not injective, and so it does not have maximal rank.  Therefore $\avc(G)$ does not have the WLP.
\end{proof}

\begin{remark}\label{rem:complementary_condition_essential}
The complementary assumption in Proposition~\ref{prop:complementary_failure}
is essential. More precisely, it cannot be replaced by the weaker assumption that
$G$ has two distinct minimum vertex covers of size $m$.
	
Indeed, the proof of Proposition~\ref{prop:complementary_failure} uses the
decomposition $V(G)=O\sqcup E$ in an essential way. This decomposition allows us
to group the squarefree monomials of degree $m-1$ according to the number of
variables coming from $O$, and then to construct an element $f\in [B_{2m}]_{m-1}$
such that
	\[
	U_{m-1}(f)=x_O+(-1)^{m-1}x_E.
	\]
	If $O$ and $E$ are merely two minimum vertex covers and are not complementary,
	then they do not define a partition of $V(G)$, and the above construction no
	longer works.
	
	Thus having two minimum vertex covers is not enough; the proof of
	Proposition~\ref{prop:complementary_failure} genuinely requires the partition
	$V(G)=O\sqcup E$.
\end{remark}
\begin{example}
Let $G$ be the graph on $[4]$ with edge set
$\{\{1,2\},\{1,3\},\{1,4\},\{2,3\}\}.$
Then $\tau(G)=2$, and the minimum vertex covers are
$\{1,2\}$ and $ \{1,3\}.$
They are not complementary, since they have nonempty intersection.  The critical map
\[\\
\cdot\ell:[\avc(G)]_1\longrightarrow [\avc(G)]_2
\]
has the matrix
\[
\begin{pmatrix}
	1&0&0&1\\
	0&1&1&0\\
	0&1&0&1\\
	0&0&1&1
\end{pmatrix}
\]
with respect to the bases
$[\avc(G)]_1=\langle x_1,x_2,x_3,x_4\rangle$
and
$[\avc(G)]_2=\langle x_1x_4,x_2x_3,x_2x_4,x_3x_4\rangle.$
Its determinant is $-2$, which is nonzero in characteristic zero.  Hence the critical map is an isomorphism, and $\avc(G)$ has the WLP. 
\end{example}

We now record a graph-theoretic characterization of the complementary-cover
condition. In particular, this proposition explains exactly when the hypothesis
of Proposition~\ref{prop:complementary_failure} can be recognized from the
graph structure.

\begin{proposition}\label{prop:complementary_characterization}
	Let $G$ be a graph with vertex set $[2m]$. The following conditions are equivalent.
	\begin{enumerate}[label=\textup{(\arabic*)}]
		\item There exist two minimum vertex covers $O$ and $E$ of $G$ such that
		$O\cap E=\emptyset$, $O\cup E=V(G)$, and $|O|=|E|=m$.
		
		\item The graph $G$ is bipartite and has a perfect matching.
		
		\item There exists a partition $[2m]=X\sqcup Y$ such that $X$ and $Y$ are
		independent sets of $G$, $|X|=|Y|=m$, and, after writing
		$X=\{x_1,\ldots,x_m\}$ and $Y=\{y_1,\ldots,y_m\}$, one has
		$\{x_i,y_i\}\in E(G)$ for every $i=1,\ldots,m$.
	\end{enumerate}
\end{proposition}

\begin{proof}
	We first prove that \textup{(1)} implies \textup{(2)}. Since $O$ is a vertex
	cover, its complement $E=[2m]\setminus O$ is an independent set. Similarly,
	since $E$ is a vertex cover, its complement $O$ is an independent set. Hence
	$G$ is bipartite with bipartition $[2m]=O\sqcup E$. Moreover, since $O$ and
	$E$ are minimum vertex covers of size $m$, we have $\tau(G)=m$. By K\"onig's
	theorem, because $G$ is bipartite, the matching number of $G$ satisfies
	$\nu(G)=\tau(G)=m$. Thus $G$ has a matching with $m$ edges. Since $G$ has
	$2m$ vertices, such a matching saturates all vertices. Hence $G$ has a perfect
	matching.
	
	Next, \textup{(2)} implies \textup{(3)}. Since $G$ is bipartite, there exists
	a bipartition $[2m]=X\sqcup Y$ such that $X$ and $Y$ are independent sets. Let
	$M$ be a perfect matching of $G$. Every edge of $M$ has one endpoint in $X$ and
	one endpoint in $Y$. Since $M$ saturates all $2m$ vertices, it contains exactly
	$m$ edges. Hence $|X|=|Y|=m$. After relabeling the vertices of $X$ and $Y$, we
	may write $X=\{x_1,\ldots,x_m\}$ and $Y=\{y_1,\ldots,y_m\}$ in such a way that
	$M=\{\{x_i,y_i\}:1\leq i\leq m\}$. Therefore $\{x_i,y_i\}\in E(G)$ for every
	$i=1,\ldots,m$.
	
	Finally, we prove that \textup{(3)} implies \textup{(1)}. Since $Y$ is
	independent, its complement $X$ is a vertex cover of $G$. Similarly, since $X$
	is independent, its complement $Y$ is a vertex cover of $G$. Thus $X$ and $Y$
	are two complementary vertex covers of size $m$. Moreover, the edges
	$\{x_1,y_1\},\ldots,\{x_m,y_m\}$ form a matching of size $m$. Every vertex cover
	must contain at least one endpoint of each edge in this matching, and therefore
	every vertex cover has size at least $m$. Hence $\tau(G)\geq m$. Since $X$ and
	$Y$ are vertex covers of size $m$, we also have $\tau(G)\leq m$. Thus
	$\tau(G)=m$, and $X$ and $Y$ are minimum vertex covers. Therefore condition
	\textup{(1)} holds, with $O=X$ and $E=Y$.
\end{proof}

\section{The WLP for Path Graphs}\label{sec:paths}
Let $P_n$ be the path graph on the vertex set $[n]=\{1,\ldots,n\}$, $n\ge 2$ and the edge set 
\[
\{\{i,i+1\}\mid 1\leq i\leq n-1\}.
\]
 Since $\tau(P_n)=\lfloor n/2\rfloor$, path graphs $P_n$ lie in the borderline case $\tau(G)=\lfloor n/2\rfloor$ studied in the previous section.

The minimum vertex covers control the first degree in which the cover ideal removes monomials.
\begin{lemma}\label{lem:minmumvertexcover}
	The following hold.
\begin{enumerate}
		\item If $n=2m+1$, then $P_{2m+1}$ has a unique minimum vertex cover, namely
		$E=\{2,4,\ldots,2m\}$.
		\item If $n=2m$, then the minimum vertex covers of $P_{2m}$ are
		$C_j=\{2,4,\ldots,2j\}\cup\{2j+1,2j+3,\ldots,2m-1\}$,
		where $0\leq j\leq m$. In particular, $P_{2m}$ has exactly $m+1$ minimum vertex covers.
	\end{enumerate}
\end{lemma}

\begin{proof}
    Since the complements of maximal independence sets are exactly minimal vertex covers, the amount of the latter andd their size can be read off the independence polynomial. From the description of the independence polynomial of paths (\cite[Proposition~3.1]{NT24b}), $P_n$ has $\binom{n+1-\lfloor \frac{n+1}{2} \rfloor}{\lfloor \frac{n+1}{2} \rfloor}$ minimal vertex covers of size $n-\lfloor \frac{n+1}{2} \rfloor$. We have
    \[
    \binom{n+1-\lfloor \frac{n+1}{2} \rfloor}{\lfloor \frac{n+1}{2} \rfloor} = \begin{cases}
        1&\text{if } n=2m+1,\\
        m+1&\text{if } n=2m.
    \end{cases}  \text{ and }  n-\lfloor \frac{n+1}{2} \rfloor =\begin{cases}
        m&\text{if } n=2m+1,\\
        m&\text{if } n=2m.
    \end{cases}
    \]
    It is routine to verify that the sets given in the statement of this lemma are vertex covers of size $m$. This concludes the proof.
\end{proof}

\begin{proposition}\label{prop:HilbertseriesPath}
For every $n\geq 2$, one has
\[
H_{\avc(P_n)}(t)=\sum_{d=0}^n \left(\binom{n}{d}-\binom{d+1}{n-d}\right)t^d.
\]
Equivalently, for every $d\geq 0$, one has
$\dim_{\Bbbk}[\avc(P_n)]_d=\binom{n}{d}-\binom{d+1}{n-d}$, where we use the convention that $\binom{a}{b}=0$ if $b<0$ or $b>a$. Furthermore, the Hilbert series $H_{\avc(P_n)}(t)$ is unimodal and that its mode is either $m$ or $m-1$. 
\end{proposition}

\begin{proof}	
By Proposition~\ref{prop:HilbertseriesofG}, one has
$\dim_{\Bbbk}[\avc(P_n)]_d=\binom{n}{d}-c_d(P_n)=\binom{n}{d}-s_{n-d}(P_n)$,
where $s_{n-d}(P_n)$ denotes the number of independent sets of $P_n$ of cardinality $n-d$. 	By \cite[Proposition~3.1]{NT24b}, $s_{n-d}(P_n)=\binom{d+1}{n-d}$.

The second assertion follows from Corollary~\ref{cor:unimodal_tau=n/2}  because $\tau(P_n)=\lfloor n/2\rfloor$. 
The result follows.
\end{proof}

We now study the multiplication map which controls the WLP of $\avc(P_n)$. Since $\tau(P_n)=m=\lfloor n/2\rfloor$, Proposition~\ref{pro:critical_map} shows that it is enough to consider the middle map
\[
\cdot\ell:[\avc(P_n)]_{m-1}\longrightarrow[\avc(P_n)]_m.
\]

We first treat odd paths. Let $n=2m+1$. Since $\tau(P_{2m+1})=m$, no squarefree monomial of degree $m-1$ is killed by $J(P_{2m+1})$. Hence $[\avc(P_{2m+1})]_{m-1}=[B_{2m+1}]_{m-1}$. In degree $m$, since $P_{2m+1}$ has a unique minimum vertex cover $E=\{2,4,\ldots,2m\}$,  the only squarefree monomial killed by $J(P_{2m+1})$ is $x_E=x_2x_4\cdots x_{2m}$. Therefore $[\avc(P_{2m+1})]_m=[B_{2m+1}]_m/\Bbbk x_E$.
Thus the critical multiplication map by $\ell$ is the composition
\[
[\avc(P_{2m+1})]_{m-1}=[B_{2m+1}]_{m-1}
\xrightarrow{\ U_{m-1}\ }
[B_{2m+1}]_m
\longrightarrow
[B_{2m+1}]_m/\Bbbk x_E=[\avc(P_{2m+1})]_m.
\]
Since $h_m=\binom{2m+1}{m}-1$ and $h_{m-1}=\binom{2m+1}{m-1}$, one has
\[
h_m-h_{m-1}=\binom{2m+1}{m}-\binom{2m+1}{m-1}-1=\frac{2}{m+2}\binom{2m+1}{m}-1>0,
\]
for all $m\ge 1.$ Thus $h_{m-1}<h_m$. Therefore the critical multiplication map by $\ell$ is maximal rank is equivalent to injectivity.
\begin{lemma}\label{lem:x_EnotbelongU_m-1}
With the notation above, one has $x_E\notin \operatorname{Im} U_{m-1}$.
\end{lemma}

\begin{proof}
Let $D=\partial/\partial x_1+\cdots+\partial/\partial x_{2m+1}$. The restriction $D:[B_{2m+1}]_m\to [B_{2m+1}]_{m-1}$ is adjoint to the up-map $U_{m-1}:[B_{2m+1}]_{m-1}\to [B_{2m+1}]_m$ with respect to the monomial inner product. Hence $\ker D=(\operatorname{Im} U_{m-1})^\perp$ in degree $m$.
	
Consider the element
$g=(x_1-x_2)(x_3-x_4)\cdots(x_{2m-1}-x_{2m})\in [B_{2m+1}]_m$.
For each $1\leq j\leq m$, we have
	\[
	\left(\frac{\partial}{\partial x_{2j-1}}+\frac{\partial}{\partial x_{2j}}\right)(x_{2j-1}-x_{2j})=0.
	\]
	Moreover, $g$ does not involve $x_{2m+1}$. Since $D$ is the sum of these partial derivatives, applied factor by factor, it follows that $D(g)=0$. Thus $g\in\ker D=(\operatorname{Im} U_{m-1})^\perp$.
	
	On the other hand, the coefficient of $x_E=x_2x_4\cdots x_{2m}$ in $g$ is $(-1)^m$. Hence $\langle x_E,g\rangle=(-1)^m\neq 0$. Therefore $x_E$ is not orthogonal to $g$, whereas every element of $\operatorname{Im} U_{m-1}$ is orthogonal to $g$. Consequently, $x_E\notin\operatorname{Im} U_{m-1}$.
\end{proof}

\begin{proposition}\label{pro:odd_path_WLP}
For every $m\geq 1$, the algebra $\avc(P_{2m+1})$ has the WLP.
\end{proposition}

\begin{proof}
It is enough to prove that the critical map
$\cdot\ell:[\avc(P_{2m+1})]_{m-1}\to[\avc(P_{2m+1})]_m$ is injective.
	
Let $f\in[\avc(P_{2m+1})]_{m-1}$ and suppose that $\ell f=0$ in $[\avc(P_{2m+1})]_m$. Since 
\[
[\avc(P_{2m+1})]_{m-1}=[B_{2m+1}]_{m-1},
\]
we may regard $f$ as an element of $[B_{2m+1}]_{m-1}$. The equality $\ell f=0$ in $[\avc(P_{2m+1})]_m$ means that $U_{m-1}(f)$ is zero in $[B_{2m+1}]_m/\Bbbk x_E$. Hence there exists $\lambda\in\Bbbk$ such that $U_{m-1}(f)=\lambda x_E$ in $[B_{2m+1}]_m$.
	
Since $U_{m-1}(f)\in\operatorname{Im} U_{m-1}$ and $x_E\notin\operatorname{Im} U_{m-1}$ by Lemma~\ref{lem:x_EnotbelongU_m-1}, we must have $\lambda=0$. Thus $U_{m-1}(f)=0$. By Theorem~\ref{SLP_Stanley}, the map $U_{m-1}:[B_{2m+1}]_{m-1}\to[B_{2m+1}]_m$ is injective. Hence $f=0$.
	
Therefore the critical map $\cdot\ell:[\avc(P_{2m+1})]_{m-1}\to[\avc(P_{2m+1})]_m$ is injective, and hence has maximal rank. This completes the proof.
\end{proof}

\begin{proposition}\label{pro:even_path_no_WLP}
For every $m\geq 3$, the algebra $\avc(P_{2m})$ does not have the WLP.
\end{proposition}

\begin{proof}
Let
\[
O=\{1,3,\ldots,2m-1\},
\qquad
E=\{2,4,\ldots,2m\}.
\]
These are complementary minimum vertex covers of $P_{2m}$.  Moreover $c_m(P_{2m})=m+1$.  Since $m+1\leq \Cat_m$ for $m\geq 3$, Proposition~\ref{prop:complementary_failure} implies that $\avc(P_{2m})$ does not have the WLP.
This completes the proof.
\end{proof}
Combining the odd and even cases, together with the small cases already discussed, we obtain the complete classification of the WLP for vertex cover Artinian algebras of path graphs.

\begin{theorem}\label{thm:path_classifition}
The algebra $\avc(P_n)$ has the WLP if and only if $n$ is odd or $n\leq 5$.
\end{theorem}

	

\begin{proof}
    If $n$ is odd, then $\avc(P_n)$ has the WLP by Proposition~\ref{pro:odd_path_WLP}. If $n=2$, then $\avc(P_2)\cong \Bbbk$, and hence $\avc(P_2)$ has the WLP. If $n=4$, then $\avc(P_4)$ has the WLP by Example~\ref{exam:pathP4}. Thus $\avc(P_n)$ has the WLP whenever $n$ is odd or $n\leq 5$. On the other hand, if $n=2m$ for some $m\geq 3$, then Proposition~\ref{pro:even_path_no_WLP} shows that $\avc(P_{2m})$ does not have the WLP. This concludes the proof.
\end{proof}

\section{The WLP for Cycle Graphs and Ferrers Graphs}\label{sec:cyclesFerrers}

We first start with cycle graphs.

\begin{proposition}\label{prop:even_cycle_min_covers}
	The graph $C_{2m}$ has exactly two minimum vertex covers,
	\[
	O=\{1,3,\ldots,2m-1\},
	\qquad
	E=\{2,4,\ldots,2m\}.
	\]
	In particular, $\tau(C_{2m})=m$.
\end{proposition}

\begin{proof}
	The edges $\{1,2\},\{3,4\},\ldots,\{2m-1,2m\}$ form a matching of size $m$, so every vertex cover has cardinality at least $m$.  Both $O$ and $E$ are vertex covers of cardinality $m$, hence $\tau(C_{2m})=m$.
	
	Let $C$ be a vertex cover of cardinality $m$.  It must contain exactly one endpoint from each matching edge $\{2i-1,2i\}$.  If the parity of these choices changes while going around the cycle, then for some $i$ one has $2i-1\in C$ and $2i+2\in C$ modulo $2m$, forcing $2i,2i+1\notin C$ and leaving the edge $\{2i,2i+1\}$ uncovered.  Hence all choices have the same parity, giving $O$ or $E$.
\end{proof}

\begin{theorem}\label{thm:cycle_classification}
	For every $n\geq 3$,
	$\avc(C_n)\text{ has the WLP}\quad\Longleftrightarrow\quad n\text{ is odd}.$
\end{theorem}

\begin{proof}
    If $n=2m$, where $m\geq 2$, is even, then by Proposition~\ref{prop:even_cycle_min_covers}, $C_{2m}$ has two complementary minimum vertex covers $O$ and $E$.  Also $c_m(C_{2m})=2\leq \Cat_m$ for $m\geq 2$.  Thus Proposition~\ref{prop:complementary_failure} applies and gives failure of the WLP. On the other hand, If $n=2m+1$ is odd, then $\tau(C_{2m+1})=m+1>\frac{2m+1}{2}.$ Therefore $A_c(C_{2m+1})$ has the WLP by Theorem~\ref{WLPtaularge}.
\end{proof}

Next we consider the class of bipartite cochordal graphs, also known as \emph{Ferrers graphs}. Recall that a graph is \emph{cochordal} if its complement does not have any induced cycle, except for maybe triangles $C_3$.

\begin{theorem}\label{thm:ferrers_classification}
    Let $G$ be a bipartite cochordal graph with vertex set $[2m]$. Then $A_c(G)$ has the WLP if and only if $G\in \{ K_2,P_4,C_4\}$.
\end{theorem}

\begin{proof}
    If $m\leq 2$, then $G$ can only be $K_2, P_4,$ or $C_4$. Then $A_c(G)$ has the WLP by Theorems~\ref{thm:path_classifition} and ~\ref{thm:cycle_classification}. Now we can assume that $m\geq 3$. By \cite[Lemma~5.6]{CEM}, $G$ has two complementary minimum vertex covers, and $c_m(G)\leq m+1\leq \frac{1}{m+1}\binom{2m}{m}$. Therefore Proposition~\ref{prop:complementary_failure} applies and hence $A_c(G)$ does not have the WLP, as desired.
\end{proof}

\section{Corona graphs and well-covered trees}\label{sec:corona}
\begin{definition}
	Let $H$ be a graph with vertex set $[m]$. The corona of $H$, denoted by
	$H\circ K_1$, is the graph obtained from $H$ by attaching one new pendant vertex
	to each vertex of $H$.
	
	More precisely, $H\circ K_1$ has vertex set $[2m]$ and edge set
	\[
	E(H\circ K_1)
	=
	E(H)\cup \bigl\{\{i,m+i\}:1\leq i\leq m\bigr\}.
	\]
	The vertices $m+1,\ldots,2m$ are pendant vertices, and the edges
	$\{i,m+i\}$ are called the pendant edges.
\end{definition}

The following is likely proved in \cite{Gutman1992}. We provide a proof for completion.

\begin{proposition}\label{prop:corona_covers}
	Let $H$ be a graph with vertex set $[m]$, and let $G=H\circ K_1$ be its corona,
	with vertex set $[2m]$ and pendant edges $\{i,m+i\}$ for $i=1,\ldots,m$.
	Then $\tau(G)=m$. Moreover, the minimum vertex covers of $G$ are precisely
	\[
	C_S=\{m+i:i\in S\}\cup\{i:i\notin S\},
	\]
	where $S$ runs through all independent sets of $H$.
\end{proposition}

\begin{proof}
	The pendant edges $\{i,m+i\}$, $i=1,\ldots,m$, form a matching of size $m$.
	Hence every vertex cover of $G$ has cardinality at least $m$. On the other hand,
	the set $[m]$ is a vertex cover of $G$ of cardinality $m$. Therefore
	$\tau(G)=m$.
	
	Let $C$ be a minimum vertex cover of $G$. Since $C$ has cardinality $m$ and must
	cover all pendant edges $\{i,m+i\}$, it contains exactly one endpoint from each
	pendant edge. Define $S=\{i\in [m]:m+i\in C\}$. Then
	\[
	C=\{m+i:i\in S\}\cup\{i:i\notin S\}.
	\]
	We claim that $S$ is independent in $H$. Indeed, if $S$ contained an edge
	$\{i,j\}\in E(H)$, then $m+i,m+j\in C$, and hence $i,j\notin C$. Thus the edge
	$\{i,j\}$ would not be covered by $C$, a contradiction. Hence $S$ is independent
	in $H$.
	
	Conversely, let $S$ be an independent set of $H$. Then $C_S$ contains exactly
	one endpoint of each pendant edge $\{i,m+i\}$, so it covers all pendant edges.
	Moreover, if $\{i,j\}\in E(H)$, then $S$ cannot contain both $i$ and $j$.
	Therefore at least one of $i,j$ does not belong to $S$, and hence at least one
	of $i,j$ belongs to $C_S$. Thus $C_S$ covers every edge of $H$. Consequently,
	$C_S$ is a vertex cover of $G$ of cardinality $m$, and hence it is a minimum
	vertex cover.
\end{proof}

\begin{proposition}\label{prop:corona_independence}
	Let $H$ be a graph with vertex set $[m]$, and let $G=H\circ K_1$ be its corona,
	with vertex set $[2m]$ and pendant edges $\{i,m+i\}$ for $i=1,\ldots,m$.
	Let
	\[
	I(H;t)=\sum_{d\geq 0}s_d(H)t^d
	\]
	be the independence polynomial of $H$. Then $s_m(G)=I(H;1)$, and
	\[
	\dim_{\Bbbk}[\avc(G)]_m-\dim_{\Bbbk}[\avc(G)]_{m-1}
	=
	\operatorname{Cat}_m-I(H;1).
	\]
\end{proposition}

\begin{proof}
	Let $U$ be an independent set of $G$, and put $S=U\cap [m]$. Since the induced
	subgraph of $G$ on $[m]$ is $H$, the set $S$ is independent in $H$.
	
	Conversely, suppose that $S$ is an independent set of $H$. Once $S$ is fixed,
	we cannot choose the pendant vertex $m+i$ for any $i\in S$, because
	$\{i,m+i\}$ is an edge of $G$. On the other hand, for every $j\notin S$, the
	pendant vertex $m+j$ may be chosen independently. Therefore the independent
	sets of $G$ whose intersection with $[m]$ is $S$ contribute
	$t^{|S|}(1+t)^{m-|S|}$ to the independence polynomial. 
	
	There is a bijection between independent
	sets of $H$ and independent sets of $G$ of cardinality $m$, and consequently
	\[
	s_m(G)=|\operatorname{Ind}(H)|=I(H;1),
	\]
	the number of all independent sets of $H$. 
	
	Finally, by Proposition~\ref{prop:HilbertseriesofG}, we have
	\[
	\dim_{\Bbbk}[\avc(G)]_m=\binom{2m}{m}-s_m(G)
	=
	\binom{2m}{m}-I(H;1),
	\]
	whereas
	\[
	\dim_{\Bbbk}[\avc(G)]_{m-1}=\binom{2m}{m-1},
	\]
	because $\tau(G)=m$ by Proposition~\ref{prop:corona_covers}. Therefore
	\[
	\dim_{\Bbbk}[\avc(G)]_m-\dim_{\Bbbk}[\avc(G)]_{m-1}
	=
	\binom{2m}{m}-\binom{2m}{m-1}-I(H;1).
	\]
	Since
	\[
	\binom{2m}{m}-\binom{2m}{m-1}
	=
	\operatorname{Cat}_m,
	\]
	we obtain
	\[
	\dim_{\Bbbk}[\avc(G)]_m-\dim_{\Bbbk}[\avc(G)]_{m-1}
	=
	\operatorname{Cat}_m-I(H;1).
	\]
\end{proof}

\begin{proposition}\label{prop:bipartite_corona_failure}
	Let $H$ be a bipartite graph with vertex set $[m]$, and let
	$G=H\circ K_1$ be its corona, with vertex set $[n]=[2m]$ and pendant edges
	$\{i,m+i\}$ for $i=1,\ldots,m$. Assume that $m\geq 3$ and
	$I(H;1)\leq \operatorname{Cat}_m$. Then $\avc(G)$ does not have the WLP.
\end{proposition}

\begin{proof}
	Since $H$ is bipartite, write $[m]=X\sqcup Y$ for a bipartition of $H$. Define
	\[
	O=X\cup \{m+i:i\in Y\},
	\qquad
	E=Y\cup \{m+i:i\in X\}.
	\]
	Then $O$ and $E$ are complementary subsets of $[2m]$. Moreover, $O$ and $E$ are
	minimum vertex covers of $G$. Indeed, by Proposition~\ref{prop:corona_covers},
	the minimum vertex covers of $G$ are precisely the sets
	\[
	C_S=\{m+i:i\in S\}\cup\{i:i\notin S\},
	\]
	where $S$ runs through all independent sets of $H$. Since $X$ and $Y$ are
	independent sets of $H$, we have
	\[
	E=C_X
	\qquad\text{and}\qquad
	O=C_Y.
	\]
	Thus $O$ and $E$ are complementary minimum vertex covers of $G$.
	
	By Proposition~\ref{prop:corona_independence}, we have
	\[
	\dim_{\Bbbk}[\avc(G)]_m-\dim_{\Bbbk}[\avc(G)]_{m-1}
	=
	\operatorname{Cat}_m-I(H;1).
	\]
	Hence the assumption $I(H;1)\leq \operatorname{Cat}_m$ implies
	$\dim_{\Bbbk}[\avc(G)]_{m-1}\leq \dim_{\Bbbk}[\avc(G)]_m$. Therefore maximal
	rank for the critical map
	$\ell:[\avc(G)]_{m-1}\to[\avc(G)]_m$ is equivalent to injectivity.
	
	Since $G$ has two complementary minimum vertex covers $O$ and $E$ of size $m$,
	Proposition~\ref{prop:complementary_failure} applies. It shows that the
	critical map is not injective. Hence $\avc(G)$ does not have the WLP.
\end{proof}

\begin{corollary}\label{cor:connected_bipartite_corona_failure}
	Let $H$ be a connected bipartite graph with vertex set $[m]$, where $m\geq 3$,
	and let $G=H\circ K_1$ be its corona, with vertex set $[n]=[2m]$. Then
	$\avc(G)$ does not have the WLP.
\end{corollary}

\begin{proof}
	By Proposition~\ref{prop:bipartite_corona_failure}, it is enough to prove that
	$I(H;1)\leq \operatorname{Cat}_m$.
	
	If $m=3$, then $H\cong P_3$, since $H$ is connected and bipartite. Hence
	$I(H;1)=5=\operatorname{Cat}_3$.
	
	Now assume that $m\geq 4$. Since $H$ is connected, it has a vertex $v$ of
	positive degree. Every independent set of $H$ either avoids $v$ or contains
	$v$. There are at most $2^{m-1}$ subsets of $[m]$ avoiding $v$. If an
	independent set contains $v$, then it contains no neighbor of $v$. Since $v$
	has at least one neighbor, there are at most $2^{m-2}$ such independent sets.
	Therefore
	\[
	I(H;1)\leq 2^{m-1}+2^{m-2}=3\cdot 2^{m-2}.
	\]
	On the other hand, $\operatorname{Cat}_4=14>12=3\cdot 2^2$, and
	\[
	\frac{\operatorname{Cat}_{r+1}}{\operatorname{Cat}_r}
	=
	\frac{2(2r+1)}{r+2}>2
	\]
	for every $r\geq 4$. It follows by induction that
	$\operatorname{Cat}_m>3\cdot 2^{m-2}$ for every $m\geq 4$. Hence
	$I(H;1)<\operatorname{Cat}_m$ for every $m\geq 4$.
	
	Thus $I(H;1)\leq \operatorname{Cat}_m$ for all $m\geq 3$, and
	Proposition~\ref{prop:bipartite_corona_failure} applies. Therefore $\avc(G)$
	does not have the WLP.
\end{proof}

\begin{corollary}\label{cor:well_covered_trees}
	Let $T$ be a well-covered tree with at least two vertices. Then $\avc(T)$ has the WLP if and only if $T\cong K_2$ or $T\cong P_4$.
\end{corollary}

\begin{proof}
By the standard characterization of well-covered trees
\cite{LM99}, every well-covered tree with at least two vertices is
of the form $H\circ K_1$ for some tree $H$. We distinguish three cases.

If $|V(H)|=1$, then $T\cong K_2$. In this case
$\avc(T)\cong \Bbbk$, and hence $\avc(T)$ has the WLP.

If $|V(H)|=2$, then $H\cong K_2$ and therefore $T\cong K_2\circ K_1\cong P_4$.
Thus $\avc(T)$ has the WLP by Theorem~\ref{thm:path_classifition}.

Finally, assume that $|V(H)|\geq 3$. Since $H$ is a tree, it is connected and
bipartite. Hence Corollary~\ref{cor:connected_bipartite_corona_failure}
applies to $T=H\circ K_1$, and so $\avc(T)$ does not have the WLP.
\end{proof}

\end{document}